\documentclass[11pt]{amsart}

\usepackage{amsmath,amssymb,mathtools}
\usepackage{booktabs}
\usepackage[colorlinks=true,citecolor=blue,linkcolor=blue,urlcolor=blue]{hyperref}
\usepackage{mathtools}

\newtheorem{theorem}{Theorem}[section]
\newtheorem{proposition}[theorem]{Proposition}
\newtheorem{lemma}[theorem]{Lemma}

\theoremstyle{remark}
\newtheorem{remark}[theorem]{Remark}

\newcommand{\D}{\mathbb D}

\DeclarePairedDelimiter{\abs}{\lvert}{\rvert}
\DeclarePairedDelimiter{\norm}{\lVert}{\rVert}

\title[Moment duality and Korenblum's constant]
{Moment duality and an improved lower bound for Korenblum's constant}

\author{Frank Wikström}
\address{Centre for Mathematical Sciences, Lund University,
Box 118, SE-221 00 Lund, Sweden}
\email{frank.wikstrom@math.lth.se}

\subjclass[2020]{Primary 30H20; Secondary 30C80, 65G30}
\keywords{Bergman space, Korenblum's maximum principle,
Korenblum's constant, moment inequalities, interval arithmetic}

\begin{document}

\begin{abstract}
We introduce a moment-duality method for Korenblum's maximum principle
in the Bergman space $A^2(\D)$.  Starting from an annular coefficient
estimate of Wang, we show that admissibility of a constant~$c$ follows
from the existence of a probability measure on $[c^2,1]$ whose ordinary
and weighted moments lie on opposite sides of the Bergman moments
$1/(k+1)$.  This converts the norm comparison into a positive moment
problem.  We then give an explicit measure, consisting of eight atoms
with rational data and Lebesgue measure on a terminal interval, for
which the required inequalities admit a rigorous ball-arithmetic
certificate.  Consequently,
\[
    c_2\geq 0.4263,
\]
improving Wang's recent lower bound $c_2\geq0.3554$.
\end{abstract}

\maketitle

\section{Introduction}

Let $\D$ be the open unit disk, and let $dA(z)=\pi^{-1}\,dx\,dy$
denote normalized area measure.  The Bergman space $A^2(\D)$ consists
of all holomorphic functions~$f$ on $\D$ such that
\[
    \norm{f}_{A^2}^2 =\int_{\D} \abs{f(z)}^2\,dA(z) <\infty.
\]
Korenblum's maximum principle asserts that there exists an absolute
constant $c\in(0,1)$ with the following property: if $f,g\in A^2(\D)$
and
\[
    \abs{f(z)} \leq \abs{g(z)},\qquad \text{for $c < \abs{z} < 1$},
\]
then $\norm{f}_{A^2}\leq \norm{g}_{A^2}$.  The supremum of
the admissible constants $c$ is denoted by $c_2$ and is called
Korenblum's constant.

The maximum principle was conjectured by Korenblum \cite{Korenblum1991} and
proved by Hayman~\cite{Hayman1999}, who showed that $c_2 > 0.04$.  Successive
improved lower estimates were obtained by Hinkkanen~\cite{Hinkkanen1999},
Schuster~\cite{Schuster2006}, and Wang~\cite{Wang2006,Wang2011}. Most recently, Wang
proved
\[
    c_2\geq 0.3554
\]
by combining a coefficient decomposition with the Möbius
pseudodistance of an annulus~\cite{Wang2025}.  The best known upper
bound is $c_2 < 0.6778994$, due to Wang~\cite{Wang2008}.

The main contribution of this paper is a dual moment criterion for
admissibility.  Wang's coefficient estimate yields a pointwise
inequality between two power series with nonnegative coefficients.  A
probability measure whose moments straddle the Bergman moments then
converts that pointwise inequality directly into the desired norm
comparison.  The criterion is independent of the particular measure
used below and reduces the search for lower bounds to a positive moment
problem.

As a quantitative consequence, we obtain the following improvement.

\begin{theorem}\label{thm:main}
Korenblum's constant satisfies
\[
    c_2\geq 0.4263.
\]
Equivalently, if $f, g \in A^2(\D)$ and $\abs{f(z)} \leq \abs{g(z)}$ for $0.4263
< \abs{z} < 1$, then $\norm{f}_{A^2} \leq \norm{g}_{A^2}$.
\end{theorem}

To indicate the mechanism, write
\[
    f(z)=\sum_{k\geq0}a_kz^k,
    \qquad
    g(z)=\sum_{k\geq0}b_kz^k.
\]
We split the sequence $\abs{a_k}^2-\abs{b_k}^2$ into its positive and negative
parts.  Wang's argument gives, on every circle in the annulus, an
inequality of the form
\[
    X(t)\leq y_0+\beta_c(\sqrt t)Y(t),
    \qquad\text{when $c^2 \leq t < 1$,}
\]
where $X$ and $Y$ have non-negative coefficients and $0\leq\beta_c<1$
in the interior of the annulus.
Instead of estimating the coefficients one at a time, we integrate this
inequality against a probability measure whose moments lie on opposite
sides of the Bergman moments $1/(k+1)$.  This reduces the proof to a
positive moment problem; the precise criterion is Lemma
\ref{lem:dual-measure}.

For the application, we use a measure of the form
\[
    \mu=\sum_{j=1}^{8}w_j\delta_{t_j}
            +\boldsymbol{1}_{[R,1]}(t)\,dt.
\]
Numerical optimization was used only to discover candidate nodes and
weights.  For the proof, they are replaced by fixed terminating
decimals and hence by exact rational numbers.  At $c=0.4263$, a short
stand-alone Arb program certifies the required inequalities for this
explicit measure using ball arithmetic.  The optimizer and all floating-point exploratory
calculations play no role in the proof.

\subsection*{Tool and computational resource disclosure}
OpenAI Codex (GPT-5.6 Sol, accessed July 2026) was used as an interactive
research assistant for mathematical brainstorming, numerical exploration of
the moment problem, and assistance in developing the interval-arithmetic
verification scripts.  The author reviewed all generated material,
independently checked the arguments and computations, and takes full
responsibility for the content of the paper.

\section{Wang's annular estimate}

We record the part of Wang's argument that will be used below.  Fix
$0 < c < 1$ and $c < r < 1$.  Define
\begin{equation}\label{eq:K-def}
    \begin{split}
    K_c(r) &= 2r\left(1+\frac{c}{r^2}\right)
        \prod_{n=1}^{\infty}
        \frac{(1+r^2c^{2n-1})(1+r^{-2}c^{2n+1})(1+c^{2n})^2}
            {(1+r^2c^{2n-2})(1+r^{-2}c^{2n})(1+c^{2n-1})^2},\\
    G_c(r) &= \frac12\left(
        1+\sqrt{\frac{1+3K_c(r)^2}{1-K_c(r)^2}}
        \right),
    \qquad
    \beta_c(r)=1-\frac1{G_c(r)}.
    \end{split}
\end{equation}
At the endpoints we use the limiting convention $\beta_c(c)=
\beta_c(1)=1$. Let
\[
    M_h(r)^2=\frac1{2\pi}\int_0^{2\pi} \abs{h(re^{i\theta})}^2\,d\theta.
\]
The following is contained in equations (8)--(9) of~\cite{Wang2025}.

\begin{proposition}[Wang]\label{prop:wang}
Suppose that $f(z)=\sum a_kz^k$ and $g(z)=\sum b_kz^k$ are holomorphic in
$\D$ and that $\abs{f} \leq \abs{g}$ in $c < \abs{z} < 1$.  Put
\[
    Y(r^2)=\sum_{\substack{k\geq1\\ \abs{a_k}<\abs{b_k}}}
       (\abs{b_k}^2-\abs{a_k}^2)r^{2k}.
\]
Then
\begin{equation}\label{eq:wang-estimate}
    Y(r^2)\leq G_c(r)\bigl(M_g(r)^2-M_f(r)^2\bigr),
    \qquad\text{for $c <r<1$}.
\end{equation}
\end{proposition}

For completeness, the coefficient step in the proof is as follows.
Let $\omega=f/g$ in the annulus.  For a suitable real number $\alpha=\alpha(r)$,
the identity
\[
 |b_k|^2-|a_k|^2
 =\frac{|b_k-\alpha a_k|^2-|a_k-\alpha b_k|^2}{1-\alpha a^2}
\]
gives
\[
Y(r^2)
 \leq\frac1{2\pi}\int_0^{2\pi}
       \frac{|g(re^{i\theta})-\alpha f(re^{i\theta})|^2}{1-\alpha a^2}
       \,d\theta.
\]
The annular pseudodistance estimate for $\omega$ bounds the last
quantity by the right-hand side of \eqref{eq:wang-estimate}; the
resulting constant is exactly $G_c(r)$ in \eqref{eq:K-def}.

\section{A dual moment lemma}

The following elementary lemma is the main structural observation.

\begin{lemma}\label{lem:dual-measure}
Fix $0 < c < 1$.  Suppose that there is a probability measure $\mu$ on
$[c^2,1]$, with $\mu(\{1\})=0$, such that
\begin{align}
 \int_{c^2}^1t^k\,d\mu(t) &\geq \frac{1}{k+1},
       && k=0,1,2,\ldots,\label{eq:upper-moments}\\
 \int_{c^2}^1\beta_c(\sqrt t)t^k\,d\mu(t) &\leq \frac{1}{k+1},
       && k=1,2,\ldots.\label{eq:lower-moments}
\end{align}
Then $c$ is admissible in Korenblum's maximum principle.
\end{lemma}

\begin{proof}
Write $d_k=\abs{a_k}^2-\abs{b_k}^2$ and set
\[
 x_k=(d_k)_+ \quad (k\geq0),
 \qquad
 y_k=(-d_k)_+ \quad (k\geq0).
\]
Define
\[
 X(t)=\sum_{k\geq0}x_kt^k,
 \qquad
 Y(t)=\sum_{k\geq1}y_kt^k.
\]
Since
\[
 M_g(r)^2-M_f(r)^2=y_0+Y(r^2)-X(r^2),
\]
Proposition \ref{prop:wang} implies
\[
 Y(r^2)\leq G_c(r)\bigl(y_0+Y(r^2)-X(r^2)\bigr).
\]
Equivalently,
\begin{equation}\label{eq:cone-inequality}
 X(t)\leq y_0+\beta_c(\sqrt t)Y(t),
 \qquad c^2<t<1.
\end{equation}
The same inequality also holds at $t=c^2$ by continuity and the convention
$\beta_c(c)=1$.  Integrating \eqref{eq:cone-inequality} against $\mu$ and using
\eqref{eq:upper-moments}--\eqref{eq:lower-moments}, we obtain
\[
\begin{split}
 \sum_{k\geq0}\frac{x_k}{k+1}
 &\leq \int X(t)\,d\mu(t)\\
 &\leq y_0+\int\beta_c(\sqrt t)Y(t)\,d\mu(t)
 \leq y_0+\sum_{k\geq1}\frac{y_k}{k+1}.
\end{split}
\]
But
\[
    \norm{f}_{A^2}^2-\norm{g}_{A^2}^2 = \sum_{k\geq0} \frac{d_k}{k+1},
\]
so the preceding inequality proves the assertion.
\end{proof}

\section{An explicit certificate measure}

Fix the numbers
\begin{equation}\label{eq:c-and-R}
 c=\frac{4263}{10000}=0.4263,
 \quad s=c^2=0.18173169,
 \quad R=\frac{99}{100}=0.99.
\end{equation}
We use a measure of the form
\begin{equation}\label{eq:mu-explicit}
 d\mu(t)=\sum_{j=1}^{8}w_j\,d\delta_{t_j}(t)
          +\boldsymbol{1}_{[R,1]}(t)\,dt,
\end{equation}
where the nodes and weights are listed in Table \ref{tab:measure}.
Every decimal in the table is interpreted as an exact terminating
decimal, not as a floating-point approximation.  Thus the measure
below is fixed independently of the procedure by which it was found.

\begin{table}[ht]
\centering
\caption{The atomic part of the measure \eqref{eq:mu-explicit}.}
\label{tab:measure}
\begin{tabular}{@{}c@{\qquad}c@{\qquad}c@{}}
\toprule
$j$ & $t_j$ & $w_j$ \\
\midrule
1 & $0.1817316900000000$ & $0.3924240372921149$ \\
2 & $0.5786106116843371$ & $0.3349420206616465$ \\
3 & $0.8163605367760075$ & $0.1555309603886983$ \\
4 & $0.9200761835179418$ & $0.0649137961077427$ \\
5 & $0.9633238309788337$ & $0.0271186307684793$ \\
6 & $0.9813028389745606$ & $0.0109963471767101$ \\
7 & $0.9881242650409469$ & $0.0036210554863853$ \\
8 & $0.9900000000000000$ & $0.0004531521182229$ \\
\bottomrule
\end{tabular}
\end{table}

Notice that $t_1=s$, $t_8=R$, and an exact decimal calculation gives
\[
 \sum_{j=1}^{8}w_j=0.9900000000000000=R.
\]
Thus the atomic part has mass $R$, the absolutely continuous part has
mass $1-R$, and $\mu$ is a probability measure supported on $[s,1]$
with no atom at $1$.

For this class of measures the two conditions in Lemma
\ref{lem:dual-measure} take a particularly useful form.  Define
\begin{align}
 A_k&=\sum_{j=1}^{8}w_jt_j^k-
             \frac{R^{k+1}}{k+1},                         \label{eq:A-k}\\
 B_k&=\int_R^1\frac{t^k}{G_c(\sqrt t)}\,dt
       +\frac{R^{k+1}}{k+1}
       -\sum_{j=1}^{8}w_j\beta_c(\sqrt{t_j})t_j^k.
                                                               \label{eq:B-k}
\end{align}
Indeed, direct integration of the Lebesgue part of \eqref{eq:mu-explicit}
shows that \eqref{eq:upper-moments} is equivalent to $A_k\geq0$.
Using $\beta_c=1-1/G_c$ similarly shows that
\eqref{eq:lower-moments} is equivalent to $B_k\geq0$.  It therefore
remains to certify
\begin{equation}\label{eq:two-families}
 A_k\geq0\quad(k\geq0),
 \qquad B_k\geq0\quad(k\geq1).
\end{equation}

\section{Rigorous verification}

All computations in this section were performed with Arb ball
arithmetic through the \texttt{python-flint} interface, using $65$
decimal digits of working precision.  The complete verification script
is included with the source of this paper as
\texttt{verify\_lower\_bound\_04263.py}.  A versioned archive containing
the script and reproduction instructions is available at
\url{https://doi.org/10.5281/zenodo.21446369}.  This script does not invoke an
optimizer or read floating-point data: the constant, cutoff, nodes, and
weights are encoded as exact rational numbers.  Its only task is to certify
the explicit inequalities stated below.

\subsection{The annular product and the quadrature}

Let $K_{c,N}$ denote the product in \eqref{eq:K-def} truncated after
$N$ factors.  For $t=r^2\geq L$, applying
$\abs{\log(1+x)} \leq x$ to the omitted factors gives
\begin{equation}\label{eq:product-tail}
 \abs[\big]{\log K_c(\sqrt t)-\log K_{c,N}(\sqrt t)}
 \leq
 \left(1+3c+2c^2+\frac{c^3+c^2}{L}\right)
 \frac{c^{2N}}{1-c^2}.
\end{equation}
Differentiating the logarithm term by term gives the corresponding
tail estimate
\begin{equation}\label{eq:derivative-tail}
 \abs[\Big]{\frac{d}{dt}\log K_c(\sqrt t)
       -\frac{d}{dt}\log K_{c,N}(\sqrt t)}
 \leq
 \left(1+c+\frac{c^3+c^2}{L^2}\right)
 \frac{c^{2N}}{1-c^2}.
\end{equation}
We use $N=28$ throughout.

Interval evaluation of the differentiated finite product, enlarged by
the error in \eqref{eq:derivative-tail}, proves that
$K_c(\sqrt t)$ is strictly increasing on $[R,0.99999]$.  The interval
was treated in three blocks, each divided into $50000$ equal
subintervals.  The certified lower bounds for its
logarithmic derivative were respectively
\begin{equation}\label{eq:derivative-bounds}
\begin{array}{c|c}
{}[R,0.999]          &3.6694461888\cdot10^{-7}\\
{}[0.999,0.9999]     &3.6647630565\cdot10^{-8}\\
{}[0.9999,0.99999]   &3.6642942322\cdot10^{-9}.
\end{array}
\end{equation}
Consequently, $1/G_c(\sqrt t)$ is decreasing on these intervals.

To bound the integral in \eqref{eq:B-k}, divide an interval
$[a,b]\subset[R,0.99999]$ into subintervals $[u_j,u_{j+1}]$.  On each
subinterval we use
\begin{equation}\label{eq:lower-quadrature}
 \int_{u_j}^{u_{j+1}}\frac{t^k}{G_c(\sqrt t)}\,dt
 \geq \frac{1}{G_c(\sqrt{u_{j+1}})}
       \frac{u_{j+1}^{k+1}-u_j^{k+1}}{k+1}.
\end{equation}
We used $30000$, $5000$, and $500$ pieces on the three intervals in
\eqref{eq:derivative-bounds}, respectively.  The remaining positive
integral over $[0.99999,1]$ was discarded.

\subsection{The first moment family}

Direct ball evaluation proves $A_k>0$ for $1\leq k\leq2183$.
The smallest relative margin occurs at $k=23$ and satisfies
\begin{equation}\label{eq:first-margin}
 \frac{A_{23}}{R^{24}/24}>6.6452117503\cdot10^{-6},
 \qquad A_{23}>2.1754156722\cdot10^{-7}.
\end{equation}
For $k\geq2184$, the atom at $t_8=R$ alone is sufficient.  Indeed,
\[
 w_8(k+1)\geq2185w_8
   =0.9901373783170365>R,
\]
and hence
\[
 \sum_{j=1}^{8}w_jt_j^k\geq w_8R^k
        >\frac{R^{k+1}}{k+1}.
\]
The case $k=0$ follows from the fact that $\mu$ is a probability
measure.  This proves the first family in \eqref{eq:two-families}.

\subsection{The second moment family}

At the first node we use $\beta_c(c)=1$.  Rigorous upper enclosures at
the remaining seven nodes are
\begin{align*}
 \beta_c(\sqrt{t_2})&<0.989685929536115,\\
 \beta_c(\sqrt{t_3})&<0.995533783756722,\\
 \beta_c(\sqrt{t_4})&<0.998127951092143,\\
 \beta_c(\sqrt{t_5})&<0.999157174953505,\\
 \beta_c(\sqrt{t_6})&<0.999573932287887,\\
 \beta_c(\sqrt{t_7})&<0.999730257446456,\\
 \beta_c(\sqrt{t_8})&<0.999773067105237.
\end{align*}
Combining these enclosures with the lower quadrature
\eqref{eq:lower-quadrature} proves $B_k>0$ for $1\leq k\leq1299$.
The smallest relative margin occurs at $k=522$:
\begin{equation}\label{eq:second-margin}
 \frac{B_{522}}{R^{523}/523}>6.0214605308\cdot10^{-6},
 \qquad B_{522}>6.0035333797\cdot10^{-11}.
\end{equation}

It remains to treat $k\geq1300$.  The same interval calculation gives
\[
 \frac1{G_c(\sqrt t)}>0.0000225,
 \qquad \text{for $0.998 \leq t \leq 0.999$}.
\]
Since $\beta_c\leq1$, we may discard the positive term
$R^{k+1}/(k+1)$ in \eqref{eq:B-k} and obtain
\begin{equation}\label{eq:second-tail}
 B_k\geq
 0.0000225(0.999-0.998)(0.998)^k
 -\sum_{j=1}^{8}w_jt_j^k.
\end{equation}
At $k=1300$, the certified comparison is
\[
 1.666810362143132\cdot10^{-9}
  >1.611068196498003\cdot10^{-9}.
\]
Thereafter the positive term in \eqref{eq:second-tail} is multiplied by
$0.998$, whereas every negative term is multiplied by at most
$R=0.99$.  Thus $B_k>0$ for every $k\geq1300$.  This completes the
proof of \eqref{eq:two-families}; Lemma \ref{lem:dual-measure} now
proves Theorem \ref{thm:main}.

\begin{remark}
The measure in Table \ref{tab:measure} was discovered by solving a
linear program on a dense grid below $R$, allowing the program to choose
its own support, and then continuously refining the resulting eight
nodes.  This exploratory optimization is not part of the certificate
and is not called by the verification script.  As $R$ tends to $1$,
additional nodes accumulate near $1$; numerical experiments place the
apparent limiting value of this discrete--continuous construction near
$c=0.42639$.  Neither observation is used in the proof.
\end{remark}


\begin{thebibliography}{99}

\bibitem{Hayman1999}
W.~K. Hayman,
\emph{On a conjecture of Korenblum},
Analysis (Munich) \textbf{19} (1999), no.~2, 195--205.

\bibitem{Hinkkanen1999}
A. Hinkkanen,
\emph{On a maximum principle in Bergman space},
J. Anal. Math. \textbf{79} (1999), 335--344.

\bibitem{Korenblum1991}
B. Korenblum,
\emph{A maximum principle for the Bergman space},
Publ. Mat. \textbf{35} (1991), no.~2, 479--486.

\bibitem{Schuster2006}
A. Schuster,
\emph{The maximum principle for the Bergman space and the M\"obius
pseudodistance for the annulus},
Proc. Amer. Math. Soc. \textbf{134} (2006), no.~12, 3525--3530.

\bibitem{Wang2006}
C. Wang,
\emph{On Korenblum's maximum principle},
Proc. Amer. Math. Soc. \textbf{134} (2006), no.~7, 2061--2066.

\bibitem{Wang2008}
C. Wang,
\emph{Domination in the Bergman space and Korenblum's constant},
Integral Equations Operator Theory \textbf{61} (2008), no.~3, 423--432.

\bibitem{Wang2011}
C. Wang,
\emph{Some results on Korenblum's maximum principle},
J. Math. Anal. Appl. \textbf{373} (2011), no.~2, 393--398.

\bibitem{Wang2025}
C. Wang,
\emph{A lower bound on Korenblum's constant},
Proc. Amer. Math. Soc. \textbf{153} (2025), no.~3, 1209--1214.

\end{thebibliography}
\end{document}